%% file: PFR_groups_main.tex
\documentclass[12pt,a4paper]{amsart}

\usepackage{amsmath,amsfonts,amssymb,amsthm,amscd}
\usepackage[all,cmtip]{xy}

\theoremstyle{plain}

\newtheorem{thm}{Theorem}[section]
\newtheorem{cor}[thm]{Corollary}
\newtheorem{lem}[thm]{Lemma}
\newtheorem{prop}[thm]{Proposition}

\newtheorem{thmABC}{Theorem}

\theoremstyle{definition}

\newtheorem{defn}[thm]{Definition}
\newtheorem{rmk}[thm]{Remark}
\newtheorem{question}[thm]{Question}

\numberwithin{equation}{section}

\usepackage{color}

\newcommand{\norm}[1]{{\lvert #1 \rvert}}

\newcommand{\Sym}{\mathrm{Sym}}

\DeclareMathOperator{\Irr}{Irr}
\newcommand{\Irrabs}{\Irr^{\star}}
\newcommand{\rabs}{r^{\star}}
\newcommand{\FF}{\mathbb{F}}
\newcommand{\ZZ}{\mathbb{Z}}
\newcommand{\NN}{\mathbb{N}}
\newcommand{\llbracket}{[\![}
\newcommand{\rrbracket}{]\!]}
\DeclareMathOperator{\GL}{GL}
\DeclareMathOperator{\Hom}{Hom}
\DeclareMathOperator{\res}{Res}

\DeclareMathOperator{\Out}{Out}
\DeclareMathOperator{\Alt}{Alt}
\newcommand{\cGrAlg}[1]{\widehat{\mathbb{Z}}\llbracket#1\rrbracket}

\title[Positively finitely related groups]{Positively finitely related \\ profinite groups}

\author{Steffen Kionke} \address{Mathematisches Institut der
  Heinrich-Heine-Universit\"at, Universit\"atsstr.\ 1, 40225
  D\"usseldorf, Germany}\email{steffen.kionke@uni-duesseldorf.de}

\author{Matteo Vannacci} \address{Mathematisches Institut der
  Heinrich-Heine-Universit\"at, Universit\"atsstr.\ 1, 40225
  D\"usseldorf, Germany}\email{matteo.vannacci@uni-duesseldorf.de}

\keywords{profinite group, presentations of profinite groups, subgroup growth, representation growth, asymptotic group theory}

\subjclass[2010]{Prim.\ 20E18; Sec.\ 20E07, 20E22, 20F69}

\begin{document}

\begin{abstract}
  We define and study the class of \emph{positively finitely related} (PFR) profinite groups.
  Positive finite relatedness is a probabilistic property of profinite groups which provides a first step to defining
   higher finiteness properties of profinite groups which generalize the positively finitely generated groups introduced by Avinoam Mann.
   We prove many asymptotic characterisations of PFR groups, for instance we show the following: a finitely presented profinite group is PFR if and only if it has at most exponential representation growth, uniformly over finite fields (in other words: the completed group algebra has polynomial maximal ideal growth). From these characterisations we deduce several structural results on PFR profinite groups.
\end{abstract}

\maketitle

\input{PFR_groups_intro}

\input{PFR_groups_prelim}

\input{PFR_groups_section1}

\input{PFR_groups_uffrg}

\section*{Acknowledgements}
The first author was financed by the DFG grant KL 2162/1-1. The second author recognises financial support from Heinrich-Heine Universit\"at D\"{u}sseldorf.

\bibliographystyle{abbrv}
\bibliography{bibliography}

\end{document}

%% file: PFR_groups_intro.tex
\section{Introduction}

A profinite group $G$ equipped with its normalized Haar measure provides a probability space. Asking probabilistic questions in the realm of profinite groups often leads to interesting phenomena. A striking example is the idea, introduced by Avinoam Mann, of \emph{positively finitely generated} (PFG) profinite groups.
Mann defines a profinite group $G$ to be PFG if for some integer $k$ the probability $P(G,k)$ that $k$ random elements generate $G$ topologically is positive (for details we refer to \cite{Mann1996} and Chapter 11 of \cite{SubgroupGrowth}).
A celebrated result of Mann-Shalev \cite{MannShalev1996} states that a profinite group is PFG if and only if it has polynomial maximal subgroup growth.
Ergo, the probabilistic property of positive finite generation is equivalent to an asymptotic property.
In this article we propose to consider positive finite generation to be a robust version of finite generation. 
We set out to define, in a similar probabilistic vein, higher ``positive'' finiteness properties of profinite groups.

 Let $f:H\to G$ be a continuous homomorphism of profinite groups. We say that $f$ is \emph{positively finitely related} (PFR) if $\ker(f)$ is positively finitely normally generated in $H$. As above this means, that for some integer $k$ the probability that $k$ random elements and their $H$-conjugates generate $\ker(f)$.
In this work we initiate a thorough investigation of the following fundamental concept.
\begin{defn}
 A finitely generated profinite group $G$ is said to be \emph{positively finitely related} (PFR) if every continuous epimorphism $f:H \twoheadrightarrow G$ from any finitely generated profinite group $H$ is PFR.
\end{defn}
We will see below that it is sufficient to verify this property for all epimorphisms from finitely generated free profinite groups. From this perspective
it becomes apparent that PFR is actually a condition on the \emph{profinite} presentations and relations of $G$.
For the moment we will just mention that non-abelian free profinite groups are not PFR, whereas finitely presented pro-$p$ groups are PFR.

%
%
%

Our main result is a characterisation of positively finitely related groups in terms of asymptotic properties.
\begin{thmABC}\label{thm:A}
 Let $G$ be a finitely presented profinite group. Then the following are equivalent:
 \begin{enumerate}
  \item\label{thmA:PFR} $G$ is positively finitely related.
  \item\label{thmA:PMEG} $G$ has polynomial minimal extension growth.
  \item\label{thmA:UBERG} $G$ has uniformly bounded exponential representation growth over finite fields.
  \item\label{thmA:groupAlgebra} The completed group algebra $\cGrAlg{G}$ has polynomial maximal left ideal growth.
 \end{enumerate}
\end{thmABC}
For the definition of \emph{polynomial minimal extension growth} we refer to Section \ref{sec:minimalExtensionGrowth1} (Def.~\ref{def:PMEG}). 
Our main tool will be the characterisation in terms of representation growth.
Given a prime number $p$ and a positive integer $k$, we denote by $r_k(G,\FF_p)$ the number of irreducible continuous $k$-dimensional representations
of $G$ over $\FF_p$. The group $G$ is said to have \emph{uniformly bounded exponential representation growth} if there is a positive constant $e > 0$ such that 
$r_k(G,\FF_p) < p^{ek}$ for all primes $p$ and $k\in\NN$. We investigate this property in detail in Sections \ref{sec:UBERG1} and \ref{sec:UBERG2}.
The representation growth of groups over the field of complex numbers was investigated by several authors (see \cite{LubotzkyMartin2004, LarsenLubotzky2008, AKOV2013} and references therein). We think it is quite remarkable that the representation growth over finite fields plays an important role in the present context.

From this characterisation we deduce a few properties of PFR groups.
\begin{thmABC}\label{thm:B} The following properties hold.
\begin{enumerate}
\item\label{point:PFRopen} A profinite group is PFR if and only if some (and then any) open subgroup is PFR.
\item A finitely presented quotient of a PFR group is PFR.
\item Direct products of PFR groups are PFR.
\end{enumerate}
\end{thmABC}
It is noteworthy that the same properties hold for PFG groups. For positive finite generation the virtual invariance is a deep result of Jaikin-Zapirain and Pyber \cite{JaikinPyber2011}. Whereas in our case \eqref{point:PFRopen} follows relatively easily from Theorem \ref{thm:A}.
Moreover, positive finite \emph{generation} is preserved under extentions. In addition to direct products we can also show that certain semidirect products of PFR groups are PFR. However, this is very technical and it does not add any substantial concept - so this will not be discussed here. We do not know whether being PFR is preserved under extensions in general. 

A finite group is said to be \emph{involved} in a profinite group $G$, if it is a continuous quotient of an open subgroup of $G$.
By a result of Borovik-Pyber-Shalev \cite{BorovikPyberShalev1996} a finitely generated profinite group $G$ is PFG if some finite group is not involved in $G$.
It is known that there are PFG groups which involve every finite group (e.g.~\cite{Bhattacharjee1994}).
In Section \ref{sec:UBERG2} we establish the same criterion for PFR.
\begin{thmABC}\label{thm:C}
Let $G$ be a finitely presented profinite group. If some finite group is not involved in $G$, then $G$ is PFR.
\end{thmABC}

All these results suggest that positive finite relatedness and positive finite generation are in some sense similar properties. However, we do not understand the exact relation between the two. Note that PFG does not imply PFR. Indeed, a PFR group is finitely presented, however, there are PFG groups which are not finitely presented (the groups considered in \cite{Vannacci} are PFG by \cite{Quick}). We consider the following question to be of importance in future investigations of this topic.
\begin{question}
Is every PFR group also PFG? Is every finitely presented PFG group also PFR?
\end{question}

\subsection*{Organization of the article}
For the convenience of the reader we gathered some preliminary results in Section \ref{sec:preliminaries}. We also collect some notation.

In Section \ref{sec:MinimalExtensionGrowth} we introduce profinite presentations and give the main definitions. In particular, we introduce the concept of extension growth. We prove the equivalence of
\eqref{thmA:PFR} and \eqref{thmA:PMEG} in Theorem \ref{thm:A}.

The minimal non-abelian extensions are considered in Section \ref{sec:MinimalNon-abelianExt}. The main result of this section is that minimal non-abelian extension growth can be characterised in terms of subgroup growth.

The minimal abelian extensions are closely related to the modular representation theory of profinite groups. We discuss this in
Section \ref{sec:UBERG1}. In this section we prove the equivalence of \eqref{thmA:PFR} and \eqref{thmA:UBERG} in Theorem \ref{thm:A}.

In Section \ref{sec:UBERG2} we use the characterisation via representation growth to complete the proof of Theorem \ref{thm:A} and to prove Theorems \ref{thm:B} and \ref{thm:C}. We also discuss the example of the free profinite groups in detail.

In the last section we briefly introduce and discuss a weaker version of the PFR property. We explain that it is a rather fragile and mysterious property.

%% file: PFR_groups_prelim.tex
\section{Preliminaries}\label{sec:preliminaries}

\subsection{Notation and conventions}
The main focus of this work are profinite groups,
therefore all subgroups of profinite groups will be intended
to be closed and all homomorphisms between profinite groups to be continuous unless explicitly stated.
In particular, a subset $X$ of a profinite group $G$ is said to generate $G$, if no proper closed subgroup of $G$ contains $X$.
Throughout $d(G)$ denotes the minimal cardinality of a generating set of the profinite group $G$.

The symbols $\NN$ and $\ZZ$ denote the set of natural numbers and the ring of integers respectively.
For a prime power $q$, the finite field with $q$ elements will be denoted by $\FF_q$. The algebraic closure of a field $k$ will be denoted by $\bar{k}$.

\subsection{Probability theory}
The following result is a standard tool in probability theory and its easy proof can be found in \cite[pp.\ 389-392]{Renyi1970}. 

\begin{lem}(Borel-Cantelli)\label{lem:borelcantelli}
Let $X_i$ be a sequence of events in a probability space $X$ with probabilities $p_i$.
\begin{enumerate}
 \item If $\sum p_i$ converges, then the probability that infinitely many of the $X_i$ happen is $0$.
 \item\label{it:BorelCantelli2} If the $X_i$ are pairwise independent and $\sum p_i$ diverges, then the probability that infinitely many of the $X_i$ happen is $1$.
\end{enumerate}
\end{lem}

\subsection{Automorphisms of direct powers of a simple group}
Let $G$ be a finite group. We denote by $\mathrm{Aut}(G)$ the automorphism group of $G$ and by $\mathrm{Out}(G)$ the group of outer automorphisms of $G$. 
\begin{lem}\label{lem:automorphismsSk}
  Let $S$ be a finite non-abelian simple group and fix an integer $k$. Then $\mathrm{Aut}(S^k)$ is isomorphic to $\mathrm{Aut}(S)^k\rtimes \mathrm{Sym}(k)$ with $\mathrm{Sym}(k)$ naturally permuting the $k$ copies of $\mathrm{Aut}(S)$. 
  \begin{proof}
  	The $k$ direct factors of $S^k$ are the only minimal normal subgroups of $S^k$. Any automorphism $\varphi$ of $S^k$ will then permute these minimal normal subgroups. This gives a surjective homomorphism $\psi$ from $\mathrm{Aut}(S^k)$ onto the symmetric group $\mathrm{Sym}(k)$. Identify $\mathrm{Sym}(k)$ with the subgroup of $\mathrm{Aut}(S^k)$ which only permutes the factors, then $\varphi \circ \psi(\varphi)^{-1}$ fixes each minimal normal subgroup and hence it induces automorphisms of these direct factors. Therefore $\mathrm{Aut}(S^k)$ is isomorphic to $\mathrm{Aut}(S)^k\rtimes \mathrm{Sym}(k)$ with $\mathrm{Sym}(k)$ naturally permuting the $k$ copies of $\mathrm{Aut}(S)$. Moreover, we have that $\mathrm{Out}(S^k)$ is isomorphic to $\mathrm{Out}(S)^k \rtimes \mathrm{Sym}(k)$ with the action of the symmetric group as above.
  \end{proof}
\end{lem}


\subsection{Extensions of profinite groups and cohomology}
We briefly discuss some results on extensions of profinite groups that will be of use in later sections.
For the basic terminology of group extension we refer to Section 11.1 of \cite{RobinsonBook}.

Let $G$ be a profinite group and $K$ a finite group. An \emph{extension} of $G$ by $K$ is a short exact sequence $1 \to K \xrightarrow{i} E \xrightarrow{p} G \to 1$. Two extensions $1 \to K \xrightarrow{i_i} E_i \xrightarrow{p_i} G \to 1$, $i=1,2$, of $G$ by $K$ are \emph{isomorphic} if there exist isomorphisms $\alpha: K \to K$ and $\theta : E_1\to E_2$ that make the diagram
\[
  \xymatrix{ 1 \ar[r] & K \ar[r]^{i_1} \ar[d]_{\alpha} & E_1 \ar[r]^{p_1} \ar[d]_{\theta} & G \ar[r] \ar@{=}[d] & 1 \\
             1 \ar[r] & K \ar[r]^{i_2} & E_2 \ar[r]^{p_2} & G \ar[r] & 1 } 
\]
commutative. Two extensions are \emph{equivalent} if they are isomorphic with $\alpha = \mathrm{id}_K$. The following lemma is a generalization of \cite[11.4.21]{RobinsonBook} to profinite groups. This is of some interest, but we do not know of an explicit reference.

\begin{lem}\label{lem:coupling}
	Let $K$ be a finite group with trivial centre and let $G$ be a finitely generated profinite group.
	The equivalence classes of extensions of $G$ by $K$ are in bijection with the elements of $\Hom(G, \Out(K))$ (called \emph{couplings}).
	The isomorphism classes of extensions of $G$ by $K$ are in bijection with the conjugacy classes of couplings.
\end{lem}
 For completeness we will sketch a proof of the lemma in 5 steps.
\begin{proof}
We fix a presentation $1\to R \to F \xrightarrow{\pi} G \to 1$ of $G$, where $F$ is a finitely generated free profinite group.

\noindent \textbf{Step 1.}
Let $1 \to K \xrightarrow{i_i} E_i \xrightarrow{p_i} G \to 1$, $i=1,2$, be two extensions of $G$ by $K$ with coupling $\chi_1$ and $\chi_2$ respectively. If the two extensions are isomorphic check that $\chi_2$ 
is conjugated to $\chi_1$ in $\mathrm{Out}(K)$. 

\noindent \textbf{Step 2.} Consider an extension $1 \to K \to E \xrightarrow{p} G \to 1$ of $G$ by $K$ with associated coupling $\chi: G \to \mathrm{Out(K)}$. Since $F$ is free, there exists a continuous map $\eta: F \to \mathrm{Aut}(K)$ that makes
\[
\xymatrix{ \ & F \ar[d]^{\chi\circ \pi} \ar@{-->}[ld]_\eta \\ \mathrm{Aut}(K) \ar[r] & \mathrm{Out}(K)}
\]
commute. Form the semidirect product $B= K \rtimes_\eta F$, this is sometimes called a \emph{covering group} of the extension. Clearly,
$1 \to K \to B \xrightarrow{q} F \to 1$ is an extension.
It is straightforward to verify that the equivalence class of this extension is uniquely determined by $\chi$ and its isomorphism class is uniquely determined by the conjugacy class of $\chi$. Deduce that the lemma holds if $G$ is a free group.

\noindent \textbf{Step 3.} By Step 2, for every extension $1 \to K \to E \xrightarrow{p} G \to 1$ of $G$ by $K$ with coupling $\chi$ there exists a surjective homomorphism $f: B \to E$ such that $p\circ f = \pi \circ q$ and $f$ is the identity on $K$.

\noindent \textbf{Step 4.} Check that the kernel of $f$ is precisely $C_{q^{-1}(R)}(K)$ (here we use $Z(K) = 1$).
Use this intrinsic description and Step $2$ to show that the (conjugacy class of) the coupling determines the equivalence (resp. isomorphism) class of the extension uniquely.

\noindent \textbf{Step 5.} Use the description $E = B/C_{q^{-1}(R)}(K)$ obtained in Step 4 to show that an extension with a given coupling always exists. 
\end{proof}

In Section~\ref{sec:UBERG1} we will apply the ``five term exact sequence''.
\begin{thm}\textnormal{\cite[Corollary~7.2.5]{RibesZalesskii}}\label{thm:5termsequence}
   Let $G$ be a profinite group, $K$ a closed normal subgroup of $G$ and let $A$ be a discrete $G$-module. Then there exists a five term exact sequence
   \begin{multline*}
    0 \to  H^1(G/K, A^K) \to H^1(G,A) \to H^1(K,A)^{G/K}  \to \\ H^2(G/K,A^K) \to H^2(G,A).
   \end{multline*}
\end{thm}

%% file: PFR_groups_section1.tex
\section{PFR groups and minimal extension growth}\label{sec:MinimalExtensionGrowth}


\subsection{Presentations of profinite groups}
We begin with a brief survey on \emph{profinite presentations}. We direct the reader to the excellent paper by Lubotzky \cite{lubotzky:prof_pres} for more details on this topic and we merely state the definition and a few properties below.
For our purpose we will use the following slightly restrictive definition of presentation, which allows only to consider finitely generated free profinite groups.
\begin{defn}
	Let $G$ be a finitely generated profinite group. A \emph{profinite presentation} of $G$ is a short exact sequence
	\[
	1 \longrightarrow R \longrightarrow F \longrightarrow G \longrightarrow 1
	\]
	where $F$ is a \emph{finitely generated} free profinite group.
	The profinite group $G$ will be called finitely presented, if there is a presentation such that the relation subgroup $R$ is finitely normally generated in $F$.
\end{defn}

One of the main differences between abstract and profinite presentations is described in \cite[Proposition~2.4]{lubotzky:prof_pres}. This essentially says that, 
fixing an integer $d\ge \mathrm{d}(G)$, the profinite group $G$ admits a unique profinite presentation with $d$ generators.
Moreover, it follows that if $G$ is finitely presented, then the relation subgroup $R$ of any presentation is finitely normally generated in $F$.

\subsection{Positively finitely related profinite groups}
Let $G$ be a finitely generated profinite group and fix a presentation  
\begin{equation}\label{eq:ses}
 1 \longrightarrow R \longrightarrow F \longrightarrow G \longrightarrow 1,
\end{equation}
denoted by $\mathcal{S}$, where $R$ is a normal subgroup of a suitable finitely generated free profinite group $F=F_d$. Let $k$ be a natural number and write $\mu$ for the Haar measure on $R^k$ normalized with $\mu(R^k)=1$. Define 
\begin{equation*}
  P_{\lhd}(\mathcal{S},k)= \mu\left(\left\{ (z_1,\ldots,z_k)\in R^k \ \vert \  \overline{\langle z_1,\ldots,z_k \rangle^{F}} = R \right\}\right).
\end{equation*}
The previous number could be interpreted as the ``probability that $k$ randomly chosen relations in $R$ already generate all relations of $G$''.

 We will say that a presentation $\mathcal{S}$ of the finitely generated profinite group $G$ is \emph{positively finitely related} (\emph{PFR}) if there exists $k\in \mathbb{N}$ such that $P_\lhd(\mathcal{S},k)>0$. In particular, such a group $G$ is finitely presented.

\begin{rmk}
 Note that a profinite group can have presentations which are PFR and presentations which are not. For instance, consider the following two presentations of the $2$-generated free profinite group $F_2$: the ``trivial'' $2$-generated presentation
 \begin{equation*}
  1\to 1 \to F_2 \to F_2 \to 1 \tag{$\mathcal{S}_{2}$} 
 \end{equation*}
and the $3$-generated presentation
 \begin{equation*}
  1\to R \to F_3 \to F_2 \to 1. \tag{$\mathcal{S}_{3}$} 
 \end{equation*}
Clearly $\mathcal{S}_{2}$ is trivially PFR, since $P_\lhd(\mathcal{S}_{2},1)=\mu_{\{1\}}(\{1\})=1$. However, we will see later that $P_\lhd(\mathcal{S}_{3},k)=0$ for every $k\in \mathbb{N}$.
\end{rmk}

\begin{defn}
A finitely generated profinite group is said to be \emph{positively finitely related} 
(\emph{PFR
}, for short) if each of its presentations is PFR.
\end{defn}

The following lemma shows that this definition is equivalent to the definition given in the introduction.
\begin{lem}
 A finitely generated profinite group $G$ is PFR if and only if the kernel $\ker(f)$ of any epimorphism $f\colon H \to G$ from any finitely generated profinite group $H$
 is positively finitely normally generated in~$H$.
\end{lem}
\begin{proof}
 The ``if'' direction is clear by using for $H$ only the finitely generated free profinite groups.
 Conversely, assume that $G$ is PFR and let $f\colon H \to G$ be a surjective homomorphism, where $H$ is $d$-generated for some $d\in \NN$.
 Let $F$ be the finitely generated free profinite group on $d$ generators and let $p\colon F \to H$ be a surjective homomorphism.
 The homomorphism $f \circ p\colon F \to G$ yields a presentation of $G$ and by assumption the kernel $R = \ker(f \circ p)$ is positively finitely normally generated in $F$. It is easy to see that the property of being ``positively finitely normally generated'' is closed under quotients. 
 Consequently, the kernel $\ker(f) = p(R)$ is also positively finitely normally generated in $H$. 
 \end{proof}




\subsection{Maximal $F$-stable subgroups}

In \cite{MannShalev1996} Mann and Shalev characterise PFG groups in terms of maximal subgroup growth. Here we prove a first parallel result.

 Let $H$ be a profinite group and $R$ a closed normal subgroup of $H$.
 A closed subgroup $M \leq R$ is said to be \emph{maximal $H$-stable} in $R$ if it is properly contained in $R$, normal in $H$ and maximal with respect to this condition in $R$.
 It is easy to see that every maximal $H$-stable subgroup is open in $R$ and hence of finite index.
 Let $n$ be a natural number and write $ m_n^{H}(R)$ for the number of open maximal $H$-stable subgroups of $R$ with index $n$.
 Indeed, if $H$ is finitely generated and $R$ is finitely normally generated in $H$, then the numbers $m_n^{H}(R)$ are finite.
 The subgroup $R$ is said to have \emph{polynomial maximal $H$-stable subgroup growth} (P$\text{M}_{\lhd H}$SG, for short)
 if there exists a constant $c>0$ such that $$m_n^{H}(R) \le n^c \quad \text{for every $n$}.$$

\begin{prop}\label{prop:PFRiffPMNSG}
 A presentation $1\to R\to F \to G\to 1$ of a finitely generated profinite group $G$ is PFR if and only if $R$ has polynomial maximal $F$-stable subgroup growth.
 \end{prop}
 \begin{proof}
  Fix a presentation $\mathcal{S}$ as in the statement of the proposition. Suppose first that $m_n^{F}(R) \le n^c$. 
  Write $\mathcal{M}$ for the set of maximal $F$-stable open subgroups of $R$. This is a countable set by assumption.
  
 Observe that a $k$-tuple $(z_1,\ldots,z_k)\in R^k$ fails to normally generate $R$ if and only if the elements $z_1,\ldots,z_k$ are contained in some $M \in\mathcal{M}$. 
In particular, we have
\[
    1-P_\lhd(\mathcal{S},k) \le \sum_{M\in \mathcal{M}} \frac{1}{\norm{R:M}^k} = \sum_{n\geq 2} \frac{m_n^{F}(R)}{n^k} \le \sum_{n\geq 2} n^{c-k}.
  \]
 The last sum is strictly less than one for any $k \ge k_0= \lceil c+2 \rceil$ and therefore $P_\lhd(\mathcal{S},k_0)>0$.

 

For the converse, suppose that $R$ does not have P$\text{M}_{\lhd F}$SG. We can assume that $R$ is finitely normally generated, otherwise $\mathcal{S}$ is not PFR. As before, denote by $\mathcal{M}=\{M_i \vert i\in \mathbb{N}\}$ the countable set of open maximal $F$-stable subgroups of $R$. By assumption the series $\sum_{n\in \mathbb{N}} m_n^{F}(R)/n^k$ diverges for every $k\in \mathbb{N}$. The events defined by the maximal $F$-stable subgroups are pairwise
 independent\footnote{This step is much simpler than in the setting of PFG groups.
 Indeed, for two distinct $M_i, M_j\in \mathcal{M}$ the product $M_i \cdot M_j$ is a normal subgroup in $F$ and hence equals $R$ by maximality. We conclude $[R:M_i\cap M_j]=[R:M_i][R:M_j]$.}.
By the second Borel-Cantelli lemma (Lemma~\ref{lem:borelcantelli}\eqref{it:BorelCantelli2}), it follows that for every $k$
\[
   \mu\left( \bigcap_{n\in \mathbb{N}} \bigcup_{m\ge n} M_m^k \right)= 1.
\]
 On the other hand, 
 if $k$ elements from $R$ lie in infinitely many open maximal $F$-stable subgroups of $R$, for sure the normal closure of the subgroup they generate will not be the whole of $R$. In particular,
 $$1-P_\lhd(\mathcal{S},k) \ge \mu\left( \bigcap_{n\in \mathbb{N}} \bigcup_{m\ge n} M_m^k \right)= 1$$ and $\mathcal{S}$ is not PFR.
 \end{proof}

\subsection{Counting minimal extensions}\label{sec:minimalExtensionGrowth1}
In the following a short exact sequence $1\to K \to E \to G \to 1$ of profinite groups 
will be called an \emph{extension of} $G$ by $K$.
We say that the extension is minimal, if the group $K$ is a finite minimal normal subgroup of $E$.
In this case $K$ is a characteristically simple finite group, in particular, it is isomorphic to a direct product of simple groups.
The extension is called $d$-generated, if the group $E$ can be generated by $d$ elements. The order $|K|$ will be called the \emph{degree} of the extension.

We will count the \emph{isomorphism classes} of minimal extensions of a group $G$. Recall that two extensions 
$1\to K_1 \to E_1 \stackrel{\pi_1}{\longrightarrow} G \to 1$ and $1\to K_2 \to E_2 \stackrel{\pi_2}{\longrightarrow} G \to 1$
of $G$ are isomorphic, if there is an isomorphism $f\colon E_1 \to E_2$ such that  $\pi_2 \circ f = \pi_1$.
It is important to note that this notion of isomorphism is weaker than the usual \emph{equivalence} of extensions.

Let $n$ be a natural number.
We write $e_n^{\mathrm{min}}(G)$ to denote the number of isomorphism classes of minimal extensions of degree $n$ of $G$.
For $d \geq d(G)$, we write $e_{n,d}^{\mathrm{min}}(G)$ to denote the number of
minimal $d$-generated extensions of degree $n$ of $G$.

\begin{lem}\label{lem:minimalExtensionsBasic}
 Let $1\to K \to E \to G \to 1$ be a minimal extension of a $d$-generated profinite group $G$.
 Then $E$ is $(d+2)$-generated. 
 
 If $K$ is abelian, then it is elementary abelian. In this case $E$ is $(d+1)$-generated and 
 the action of $G$ on $K$ defines an irreducible representation of $G$.
\end{lem}
\begin{proof}
As $K$ is characteristically simple, it is isomorphic to a direct product $K \cong S^k$ where $S$ is a finite simple group.
If $S$ is non-abelian, then the action of $E$ on the factors of $S^k$ is transitive. Since $S$ is $2$-generated (see \cite[Theorem~B]{AschbacherFirstCohomologyGroup}),
the group $E$ can be generated by $d+2$ elements.

Now assume that $K$ is abelian.
Since $K$ is minimal normal in $E$, it does not have any proper characteristic subgroup. Hence $K$ is elementary abelian, 
i.e.\ a finite dimensional $\mathbb{F}_p$-vector space for some prime~$p$. 
Suppose that $W$ is a proper $G$-stable subspace of $K$,
then $W$ is a normal subgroup of $E$ contained in $K$, against the minimality of~$K$. Therefore $W$ must be trivial and the action of $G$ on~$K$ irreducible.
In this case $E$ can be generated by $d+1$ elements. Indeed, one can lift a minimal generating set of $G$ to $E$ and add any non-trivial element of~$K$ to obtain a generating set for $E$.
\end{proof}

\begin{cor}\label{cor:d+2lemma}
 Let $G$ be a $d$-generated profinite group. Then
 \begin{equation*}
   e_n^{\mathrm{min}}(G) = e_{n,d+2}^{\mathrm{min}}(G) \quad \text{for all $n \in \NN$}.
 \end{equation*}
\end{cor}

\begin{defn}\label{def:PMEG}
We say that $G$ has \emph{polynomial minimal extension growth},
if there is a constant $\alpha > 0$ such that $ e_n^{\mathrm{min}}(G) \leq n^\alpha$ for all $n \in \NN$.
\end{defn}

The next result implies the equivalence of \eqref{thmA:PFR} and \eqref{thmA:PMEG} in Theorem~\ref{thm:A}.
\begin{thm}\label{thm:minimalExtensionGrowth}
 Let $G$ be a finitely generated profinite group. The following are equivalent:
 \begin{enumerate}
  \item $G$ is PFR
  \item $G$ has polynomial minimal extension growth
  \item $G$ has polynomial minimal $(d(G)+2)$-generated extension growth.
 \end{enumerate}
\end{thm}
\begin{proof}
We want to relate the minimal extensions of $G$ to the maximal $F$-stable subgroups of a given presentation.
Let $\mathcal{S}$ be the presentation
\begin{equation*}
  1\longrightarrow  R \longrightarrow F_d \longrightarrow G \longrightarrow 1 
\end{equation*}
of the finitely generated profinite group $G$.
To each open maximal $F_d$-stable subgroup $M$ of $R$ with index $n$ we can associate an extension $1 \to R/M \to F_d/M \to G \to 1$
of $G$. By construction this is a $d$-generated minimal extension of $G$ of degree $n$.

Indeed, every such extension can be obtained in this way.
Let
$1\to K \to E \to G \to 1$ be a minimal $d$-generated extension of $G$.  It follows from Gasch\"utz' Lemma (see e.g.~\cite[Lem.2.1]{lubotzky:prof_pres}) that there is a \emph{surjective} homomorphism
$f\colon F_d \to E$ such that the following diagram commutes
\begin{equation*}
\begin{CD}
 1 @>>> R @>>> F_d @>>> G @>>> 1\\
   @.  @VV{f_{|R}}V   @VV{f}V      @|     @. \\
 1 @>>> K @>>> E @>>> G @>>> 1.\\
\end{CD}
\end{equation*}
Thus $M = \mathrm{ker}(f)$ is an open maximal $F$-stable subgroup of $R$.
There are at most $|K|^d$ epimorphisms $f$ as above.
In particular, we obtain the inequalities
\begin{equation*}
 e_{n,d}^{\mathrm{min}}(G) \leq m_n^{F_d}(R) \leq  n^d e_{n,d}^{\mathrm{min}}(G).
\end{equation*}
Furthermore, by Proposition~\ref{prop:PFRiffPMNSG} the presentation $\mathcal{S}$ is PFR exactly if $G$ has polynomial minimal $d$-generated extension growth.
The theorem follows from Corollary \ref{cor:d+2lemma} and the fact that $e_{n,d}^{\mathrm{min}}(G) \leq e_{n,d+2}^{\mathrm{min}}(G)$.
\end{proof}

\section{Minimal non-abelian extensions}\label{sec:MinimalNon-abelianExt}
An extension $1\to K \to E \to G \to 1$ will be called \emph{abelian} if $K$ is abelian. Otherwise, we will call the extension \emph{non-abelian}. In this section we investigate the minimal non-abelian extensions of a finitely generated profinite group $G$.
We write $e_n^{\mathrm{min,n.a.}}(G)$ to denote the number of minimal non-abelian extensions of $G$ of degree $n$. If this sequence of numbers admits a polynomial bound in $n$, we say that $G$ has polynomial minimal non-abelian extension growth.

\begin{thm}\label{thm:non-abelianExtensionGrowth}
 A finitely generated profinite group $G$ has polynomial minimal non-abelian extension growth if and only if it has at most exponential subgroup growth.
\end{thm}
In the proof we will freely use Lemma~\ref{lem:coupling} without mention.
\begin{proof}
Let $d = d(G)$ denote the minimal number of generators of $G$.
We will denote the set of conjugacy classes of open subgroups (of index $k$) of $G$ by $\mathrm{Conj}(G)$ (resp.~by $\mathrm{Conj}_k(G)$).

Assume that $G$ has polynomial minimal non-abelian extension growth and choose any finite non-abelian simple group $S$ of order $s = |S|$.
For every open subgroup $H \leq_o G$ we construct the semidirect product $E_H = S^{G/H} \rtimes G$, where $G$ permutes the $|G/H|$ copies of $S$ as the cosets of $G/H$ by left multiplication.
Clearly, $1 \to S^{G/H} \to E_H \to G \to 1$ is a split extension and it is minimal since $G$ acts transitively on $G/H$.
Two such extensions $E_{H_1}$ and $E_{H_2}$ are isomorphic (as extenstions of $G$) exactly if $H_1$ and $H_2$ are conjugate.
Hence we obtain the inequality
\begin{equation*}
     e_{s^k}^{\mathrm{min,n.a.}}(G) \geq |\mathrm{Conj}_k(G)| \geq \frac{a_k(G)}{k}
\end{equation*}
and we conclude that $G$ has at most exponential subgroup growth.

Conversely, assume that $G$ has at most exponential subgroup growth. So there is a constant $\lambda >0$ such that $a_k(G) \leq 2^{\lambda k}$ for all $k \in \NN$.
Let $S$ denote a non-abelian finite simple group of order $s$.

We need some notation.
Let $\mathrm{Ex}_k(S,G)$ denote the set of isomorphism classes of minimal extensions of $G$ by $S^k$. We put $\mathrm{Ex}(S,G) = \bigcup_{k\in \NN} \mathrm{Ex}_k(S,G)$. By Lemma~\ref{lem:automorphismsSk}, $\Out(S^k) \cong \Out(S)^k \rtimes \Sym(k)$ with $\Sym(k)$ permuting the $k$ copies of $\mathrm{Out}(S)$.

For $k\in \NN$ we define a map $t_{S,k}\colon \mathrm{Ex}_k(S,G) \to \mathrm{Conj}_k(G)$ as follows.
Let $\mathcal{E}$ be a minimal non-abelian extension of $G$ by $S^k$ and let $\chi\colon G \to \Out(S)^k \rtimes \Sym(k)$ be the associated coupling.
Via the projection onto $\Sym(k)$ we obtain a permutation representation of $G$ on $k$ elements (the $k$ minimal normal subgroups of $S^k$).
Since the extension is minimal, the permutation is transitive. We define $t_{S,k}(\mathcal{E})$ to be the conjugacy class of the stabilizer of a point in this permutation representation. If we begin with an isomorphic extension, then the couplings are conjugate and hence we get the same permutation representation after relabelling -- in particular, the stabilizers are the same.
Using the semidirect product groups constructed above, it follows that the map $t_{S,k}$ is surjective. For an open subgroup $H\leq_o G$, we will denote by $[H]= \{H^g \vert g\in G\}$ the conjugacy class of $H$ in $G$. 

\medskip

\emph{Claim}: For every open subgroup $H \leq_o G$ of index $k$ the fibre $t_{S,k}^{-1}([H])$ has at most $|\Out(S)|^{kd}$ elements.

\medskip

Indeed, the subgroup $H$ yields a permutation representation of $G$ on $k = [G:H]$ elements and hence (enumerating the cosets of $H$) a homomorphism
$\chi_0 \colon G \to \Sym(k)$. Observe that we obtain a conjugate homomorphism, if we enumerate the cosets of $H$ differently. There are at most $|\Out(S)|^{kd}$ many couplings $\chi\colon G \to \Out(S)^k \rtimes \Sym(k)$ which extend $\chi_0$. We deduce that there are at most $|\Out(S)|^{kd}$ conjugacy classes of couplings with stabilizers in $[H]$ and the claim follows (here we use Lemma~\ref{lem:coupling}).

To conclude, we need two facts from the classification of the finite simple groups: 1) for every order there are at most two distinct finite simple groups (see \cite{ArtinOrdersofClassicalGroups}) and 2) if $S$ is a finite non-abelian simple group, then $|\Out(S)| \leq 30|\Out(S)| \leq |S|$ (see \cite[Lemma~2.2]{Quick}).
For every $n\geq 2$ we obtain
\begin{align*}
 e_{n}^{\mathrm{min,n.a.}}(G) &= \sum_{\substack{ s, k \in \NN \\ s^k = n }} \sum_{\substack{S \: \text{simple}\\ |S| = s}} |\mathrm{Ex}_k(S,G)| \\
                              &\leq \sum_{\substack{ s, k  \\ s^k = n }} \sum_{\substack{S \: \text{simple}\\ |S| = s}} |\mathrm{Conj}_k(G)| \cdot |\Out(S)|^{kd}\\
                              &\leq 2 \sum_{\substack{ s, k  \\ s^k = n }} a_k(G) \cdot n^d 
                              \leq 2 n^{d+1+\lambda}. \qedhere
\end{align*}
\end{proof}

%% file: PFR_groups_uffrg.tex
\section{Minimal abelian extensions and representation growth over finite fields}\label{sec:UBERG1}

Let $G$ denote a finitely generated profinite group.
In this section we study the minimal abelian extensions of $G$.
We will see that this is related to the asymptotic behaviour of the number of irreducible continuous representations of $G$ over finite fields.

\subsection{Definitions and notation}

Let $\FF$ be a field. We consider the class of finite dimensional representations of $G$ over $\FF$ with open kernel. Recall that a representation $(\rho,V)$ of $G$ over $\FF$ is \emph{absolutely irreducible} if it is irreducible as a representation of $G$ over the algebraic closure $\overline{\FF}$ of the field $\FF$. The set of irreducible
(resp.\ absolutely irreducible) representations of $G$ over $\FF$ will be denoted by $\Irr(G,\FF)$ (resp.~$\Irrabs(G,\FF)$). 
Let $n\geq 1$ be an integer, then $\Irr_n(G,\FF) \subseteq \Irr(G,\FF)$ (resp.~$\Irrabs_n(G,\FF)$) denotes the subset of $n$-dimensional representations.
Whenever these sets are finite -- for instance, if $\FF$ is a finite field -- we write
\begin{align*}
    r_n(G,\FF) = \left| \Irr_n(G,\FF) \right|\\
    \rabs_n(G,\FF) = \left| \Irrabs_n(G,\FF) \right|
\end{align*}
for the number of (absolutely) irreducible $n$-dimensional representations.
We also define $R_n(G,\FF) = \sum_{j=1}^n r_n(G,\FF)$.

Let $p$ be a prime. We say that the profinite group $G$ has \emph{at most exponential mod-$p$ representation growth},
if there is a constant $e > 0$ such that $r_n(G,\FF_p) \leq p^{en}$.
We are interested in the case where $e$ can be chosen uniformly for all primes.

\begin{defn}
   A finitely generated profinite group is said the have \emph{uniformly bounded exponential representation growth} (UBERG) 
   if there is a positive number $e > 0$
   such that 
   \begin{equation*}
      r_n(G,\FF_p) \leq p^{en}
   \end{equation*}
   for all primes $p$ and all positive integers $n$.
\end{defn}

Equivalently, a profinite group $G$ has UBERG exactly if there is a positive number $e'>0$ such that $R_n(G,\FF_p) \leq p^{e'n}$ for all $p$ and all $n$.

\subsection{Minimal abelian extensions and UBERG}

Remember that we denoted by $e^\mathrm{min}_n(G)$ the number of minimal extension of the profinite group $G$ with a kernel of order $n$. 
Let $p$ be a prime number. Every minimal extension of $G$ of degree $p^k$ is abelian. Conversely, it follows from Lemma \ref{lem:minimalExtensionsBasic} that every minimal abelian extension of $G$ has prime power degree.
 
\begin{prop}\label{prop:represdimpk}
 Let $p$ be a prime, $G$ a finitely generated profinite group and $k$ an integer. Then 
 \[
   r_k(G,\FF_p) \leq e^\mathrm{min}_{p^k}(G) \leq \sum_{(\rho,V_\rho)\in \Irr_k(G,\FF_p)} \norm{H^2(G,V_\rho)}.
 \]
 If $G$ is finitely presentable with $r$ relations, then $e^\mathrm{min}_{p^k}(G) \leq p^{rk}r_k(G,\FF_p)$.
 \end{prop}
 \begin{proof}
  By Lemma~\ref{lem:minimalExtensionsBasic}, any minimal extension of $G$ by a group of order $p^k$ is of the form $1\to V_\rho \to E \to G \to 1$ for some $(\rho,V_\rho)\in \mathrm{Irr}(G,\mathbb{F}_p)$ with $\dim(V_\rho) =k$.
  
   The first inequality follows from the observation that for every irreducible representation $(\rho, V_\rho)$ of $G$ over $\FF_p$, the
   semidirect product $V_\rho \rtimes G$ is a minimal extension of $G$ by $V_\rho$. The extensions are not isomorphic for distinct irreducible representations.
   
   Moreover, fixing a representation $(\rho,V_\rho)$ there are exactly $\norm{H^2(G,V_\rho)}$ different \emph{equivalence classes} of extensions of $G$ by $V_\rho$.
   This yields the second inequality.
   
    For the last part, let $1\to R \to F \to G \to 1$ be a presentation of $G$ with $r$ relations. By Theorem~\ref{thm:5termsequence} applied to the chosen presentation we obtain
 \[
   \mathrm{Hom}(R,V_\rho)^G \to H^2(G,V_\rho) \to H^2(F,V_\rho) =0.  
 \]
 Here $\mathrm{Hom}(R,V_\rho)^G$ denotes the vector space of $G$-equivariant homomorphisms from $R$ to $V_\rho$, where $G$ acts by conjugation on $R$.
Hence $\norm{H^2(G,V_\rho)} \le \vert \mathrm{Hom}(R,V_\rho)^G\vert \le \dim_{\mathbb{F}_p}(V_\rho)^r = p^{kr}$.
\end{proof}

%
%
%
%

\subsection{Subgroup growth and UBERG}
We will show that a profinite group with UBERG has at most exponential subgroup growth.

\begin{lem}\label{lem:irreducibleFaithful}
 Let $L$ be a finite group and let $(\rho, V)$ be a finite dimensional faithful representation of $L$ over some finite field $\FF_q$ of characteristic $p$.
 Suppose that $L$ has a unique minimal normal subgroup $M$ which is not a $p$-group.
 Then $L$ has a faithful irreducible $\FF_q$-representation of dimension at most $\dim_{\FF_q}(V)$.
\end{lem}
\begin{proof}
 Consider the restriction $V_{|M}$ of $V$ to $M$. 
 Let $[V_{|M}] \in G_0(\FF_q[M])$ denote the image in the Grothendieck group of finite dimensional $\FF_q[M]$-modules. Recall that $G_0(\FF_q[M])$ is a free
 abelian group generated by the simple $\FF_q[M]$-modules.
 The class $[V_{|M}]$ is not a multiple of the trivial representation. Indeed, assume this is the case, then
 there is a basis of $V$ in which $\rho(M)$ is contained in the upper triangular matrices. This is impossible since $M$ is not a $p$-group and $\rho$ is faithful.
 
 Consider the restriction homomorphism between the Grothendieck groups
 $\res_M\colon G_0(\FF_q[L]) \to  G_0(\FF_q[M])$. As $\res_M([V]) = [V_{|M}]$ is not a multiple of the trivial class,
 a composition series of $V$ as $\FF_q[L]$-module contains some irreducible factor $(\phi,W)$ on which $M$ acts non-trivially.
 By assumption $M$ is the unique minimal normal subgroup, thus $(\phi,W)$ is a faithful irreducible representation of $L$.
\end{proof}

Using this Lemma, we are able to slightly generalize Corollary 12.2 in \cite{JaikinPyber2011}.
\begin{prop}\label{prop:UBERGimpliesESG}
 Let $G$ be a finitely generated profinite group and suppose there is some prime number $p$
 such that $G$ has at most exponential representation growth over $\FF_p$.
 Then $G$ has at most exponential subgroup growth.
\end{prop}
\begin{proof}
 The argument is essentially the one in \cite[Cor.\ 12.2]{JaikinPyber2011}.
 Suppose $G$ has super-exponential subgroup growth.
 By Theorem 10.2 in  \cite{JaikinPyber2011}, for every  $c > 0$ there is finite group $L$ (depending on $c$) with a unique minimal normal subgroup $M \cong \Alt(b)^s$ for some
 $b \geq 5$, $s \geq 1$, such that $G$ has $k > c^{bs}$ distinct continuous quotients isomorphic to $L$.
 The group $L$ has a transitive permutation representation on $n = bs$ elements and thus has a faithful $\FF_p$-representation of dimension $bs$. By Lemma \ref{lem:irreducibleFaithful} we conclude that
 $L$ has a faithful irreducible $\FF_p$-representation of dimension at most $n = bs$.
 Since representations with distinct kernels are inequivalent, it follows that $R_n(G,\FF_p) \geq k > c^n$.
 As $c$ was arbitrary, the claim follows. 
\end{proof}

The next corollary is an immediate consequence of the previous proposition.

\begin{cor}
 A profinite group with UBERG has at most exponential subgroup growth.
\end{cor}

\subsection{PFR and UBERG}
We are now in the position to prove the equivalence of \eqref{thmA:PFR} and \eqref{thmA:UBERG} Theorem \ref{thm:A}.

\begin{thm}\label{thm:PFRisUBERG}
 Let $G$ be a finitely presented profinite group.
 Then G is positively finitely related (PFR) if and only if G has uniformly bounded exponential representation growth (UBERG).
\end{thm}
\begin{proof}
 Assume that $G$ is PFR. By Theorem \ref{thm:minimalExtensionGrowth} the group $G$ has polynomial minimal extension growth. The first part of Proposition \ref{prop:represdimpk} implies
 that $G$ has UBERG.
 
 Conversely, assume that $G$ has UBERG. By Proposition \ref{prop:UBERGimpliesESG} the group $G$ has at most exponential subgroup growth.
 This implies, by Theorem \ref{thm:non-abelianExtensionGrowth}, that $G$ has polynomial minimal non-abelian extension growth.
 By Theorem \ref{thm:minimalExtensionGrowth} it suffices to show that $G$ has polynomial minimal abelian extension growth.
 However, by assumption $G$ is finitely presented and has UBERG, thus Proposition \ref{prop:represdimpk} implies the claim.
\end{proof}

\begin{cor}\label{cor:PRFquotients}
 A finitely presented quotient of a PFR group is again PFR.
\end{cor}
The corollary follows from the observation that UBERG is preserved under quotients.

\section{Uniformly bounded exponential representation growth}\label{sec:UBERG2}
In this section we initiate a careful investigation of the uniformly bounded exponential representation growth.

\subsection{The completed group algebra}
Unifomly bounded exponential representation growth can be characterised in terms of the completed group algebra.
Let $G$ be a profinite group. The completed group algebra $\cGrAlg{G}$ of $G$ is the inverse limit
\begin{equation*}
   \varprojlim_{n, N }{(\ZZ/n\ZZ)[G/N]}
\end{equation*}
running over all positive integers $n \in \NN$ and all open normal subgroups $N \triangleleft_o G$.
The completed group algebra is a profinite topological ring. All ideals considered here are left ideals.
For $n \in \NN$ we denote by $m_n^{\triangleleft}(\cGrAlg{G})$ the number of maximal open ideals of index $n$ in $\cGrAlg{G}$.
As for groups, we say that $\cGrAlg{G}$ has \emph{polynomial maximal ideal growth}, if there is a constant $\gamma > 0$ such that
$m_n^{\triangleleft}(\cGrAlg{G}) \leq n^\gamma$ for all $n \in \NN$.

A profinite topological $\cGrAlg{G}$-module $M$ is said to be \emph{positively finitely generated}, if 
the probability $P_{\cGrAlg{G}}(k,M)$ that $k$ random elements generate $M$ as topological $\cGrAlg{G}$-module is positive for some $k$.
Here we use the normalized Haar measure on $M$ to define the probability.
\begin{prop}\label{prop:UBERGandPMIG}
 Let $G$ be a finitely generated profinite group. The following statements are equivalent.
 \begin{enumerate}
  \item\label{eq:UBERG1} $G$ has UBERG.
  \item\label{eq:PMIG} $\cGrAlg{G}$ has polynomial maximal ideal growth.
  \item\label{eq:PFGmodule} $\cGrAlg{G}$ is positively finitely generated as module over itself.
  \item\label{eq:PFGallModules} Every finitely generated topological $\cGrAlg{G}$-module is positively finitely generated.
 \end{enumerate}
\end{prop}
\begin{proof}
Let $I \leq \cGrAlg{G}$ be an open maximal left ideal. The quotient $\cGrAlg{G}/I$ is a finite irreducible
$G$-module over some prime field $\FF_p$. In particular, the index of $I$ is a prime power. Conversely, let $V$ be a finite irreducible $G$-module over $\FF_p$. The finite abelian group $V$ 
is a module under the completed group algebra $\cGrAlg{G}$. 
Observe that $V$ is isomorphic to $\cGrAlg{G}/I$ if and only if $I$ is the annihilator of some non-zero vector in $V$.
The equivalence of \eqref{eq:UBERG1} and \eqref{eq:PMIG} follows immediately from the inequalities
\begin{equation*}
      r_n(G,\FF_p) \leq m_{p^n}^{\triangleleft}(\cGrAlg{G}) \leq p^n r_n(G,\FF_p)
\end{equation*}
for all primes $p$ and $n \in \NN$.

Suppose now that $\cGrAlg{G}$ has polynomial maximal ideal growth and $m_n^{\triangleleft}(\cGrAlg{G}) \leq n^\gamma$.
Let $k \geq 2+\gamma$ be an integer.
We obtain
\begin{equation*}
    P_{\cGrAlg{G}}(k,\cGrAlg{G}) \geq 1 - \sum_{n = 2}^\infty \frac{m_n^{\triangleleft}(\cGrAlg{G})}{n^k} \geq 1-  \sum_{n = 2}^\infty \frac{1}{n^2} > \frac{1}{3}.
\end{equation*}

Finally, assume that $\cGrAlg{G}$ is positively finitely generated. Observe that two open maximal ideals $I,J \leq \cGrAlg{G}$ with 
non-isomorphic quotients define independent events in the Haar measure.
Now, as in the proof of Proposition \ref{prop:PFRiffPMNSG} the Borel-Cantelli Lemma~\ref{lem:borelcantelli} implies that $G$ has UBERG.

The equivalence of \eqref{eq:PFGmodule} and \eqref{eq:PFGallModules} is immediate, since positive finite generation is preserved under
direct sums and quotients.
\end{proof}

\subsection{UBERG is virtually invariant}
We recall that positive finite generation (PFG) is a property of the commensurability class. This is a deep result of 
Jaikin-Zapirain and Pyber \cite{JaikinPyber2011}.
The purpose of this section is to show that also UBERG is a property of the commensurability class of a profinite group.

\begin{thm}\label{thm:openUBERG}
 Let $G$ be a profinite group and $H \leq_o G$ any open subgroup.
 Then $G$ has UBERG if and only if $H$ has UBERG.
\end{thm}
\begin{cor}
A profinite group is PFR if and only if some (and then any) open subgroup is.
\end{cor}
The corollary follows immediately from Theorem \ref{thm:PFRisUBERG} and the fact that a profinite group is finitely presented exactly if some open subgroup is.

\begin{proof}[Proof of Theorem \ref{thm:openUBERG}]
Suppose that $H$ has UBERG. The completed group algebra $\cGrAlg{G}$ is a finitely generated $\cGrAlg{H}$-module.
By Proposition \ref{prop:UBERGandPMIG} it is positively finitely generated as $\cGrAlg{H}$-module. A fortiori it is positively finitely generated as 
$\cGrAlg{G}$-module.

Conversely, suppose that $G$ has UBERG. By Proposition \ref{prop:UBERGandPMIG} the ring $\cGrAlg{G}$ is positively finitely generated as module over itself.
As $\cGrAlg{H}$-module $\cGrAlg{G}$ is free of rank $r = [G:H]$, it suffices to show that $\cGrAlg{G}$ is positively finitely generated as $\cGrAlg{H}$-module. In fact, this would imply that every finitely generated $\cGrAlg{H}$-module is positively finitely generated.

Let $k \geq 1$ be an integer such that the probability $P_{\cGrAlg{G}}(k,\cGrAlg{G})$ is positive.
Fix a transversal $T = \{t_1,\dots,t_r\}$ of $H$ in $G$. If $(x_1,\dots,x_k)$ is a generating $k$-tuple of $\cGrAlg{G}$ as $\cGrAlg{G}$-module, 
then the $kr$-tuple $(t_1x_1, \dots, t_1x_k, t_2x_1, \dots, t_rx_k)$ is generating for $\cGrAlg{G}$ as $\cGrAlg{H}$-module.
From this observation it follows readily that $P_{\cGrAlg{H}}(k r,\cGrAlg{G}) > 0$.
\end{proof}

\subsection{UBERG is preserved under direct products}
In this section we establish the following results.
\begin{thm}\label{thm:directProducts}
Let $G_1, G_2$ be profinite groups with UBERG, then the direct product $G_1 \times G_2$ has UBERG.
\end{thm}
\begin{cor}
 PFR is preserved under direct products.
\end{cor}
\begin{rmk}
It is possible to extend this result to certain semidirect products under technical assumptions. 
\end{rmk}

For the proof we first reformulate the UBERG property in terms of absolutely irreducible representations. 
Let $G$ be a finitely generated profinite group, let $k$ be a finite field and let $\bar{k}$ denote the algebraic closure. 
The absolute Galois group
 $\mathrm{Gal}(\bar{k}/k)$ acts on the set of absolutely irreducible representations $\Irr(G,\bar{k})$ (from the right, say) and every representation has a finite orbit.
\begin{lem}\label{lem:bijectionGaloisOrbits}
Let $G$ be a profinite group and let $k$ be a finite field.
There is a bijection
\begin{equation*}
  \Phi \colon  \Irr(G,\bar{k})/\mathrm{Gal}(\bar{k}/k) \longrightarrow \Irr(G,k) 
\end{equation*}
such that 
$\dim_k \Phi( M \mathrm{Gal}(\bar{k}/k)) = | M  \mathrm{Gal}(\bar{k}/k)| \dim_{\bar{k}} M$
for all $M \in \Irr(G,\bar{k})$.
\end{lem}
\begin{proof}
Without loss of generality, we assume that $G$ is a finite group.
Let $(\rho,V)$ be an irreducible representation of $G$ over $k$. The base change representation on $V_{\bar{k}} = \bar{k} \otimes_k V$
is completely reducible (see \cite[(7.11)]{CurtisReiner1}). Fix any irreducible summand $(\vartheta,M)$ of $V_{\bar{k}}$ and let $k(\vartheta)$ denote the trace field of $(\vartheta, M)$, i.e.\ the field generated by the character values of $\vartheta$. 
The representation $\vartheta$ can be realized over the trace field $k(\vartheta)$ since the Schur index equals $1$ \cite[Thm.~1.1]{Fein1967}. Moreover,
by \cite[Thm.~1.4]{Fein1967} or \cite[(7.11)]{CurtisReiner1} we have
\begin{equation}\label{eq:decomposition1}
     V_{\bar{k}} = \bigoplus_{\sigma \in \mathrm{Gal}(k(\vartheta)/k)} M^\sigma.
\end{equation}
The map $\Phi$ is defined as follows:
Take an irreducible representation $(\vartheta, M)$ of $G$ over $\bar{k}$, choose 
an irreducible representation $(\rho, V)$ of $G$ over $k$ such that $M$ is a direct summand of $V_{\bar{k}}$ and define
$\Phi(M \mathrm{Gal}(\bar{k}/k))$ to be the isomorphism class of $(\rho,V)$. By \eqref{eq:decomposition1} (and the fact that characters of distinct simple modules are linearly independent \cite[(17.3)]{CurtisReiner1}) the representation $V$ is unique and independent of the chosen orbit representative $M$.
Since there is an obvious inverse map, $\Phi$ is bijective.
We observe that, given an irreducible representation $(\rho, V)$ of $G$ over $k$, then $\Phi^{-1}(V)$ is the orbit $\{M^\sigma \:|\: \sigma \in \mathrm{Gal}(\bar{k}/k)\}$ which has precisely $[k(\vartheta): k]$ elements. 
\end{proof}

\begin{lem}\label{lem:absolutely}
 A profinite group has UBERG if and only if there is a positive constant $b > 0$ such that
 \begin{equation*}
      \rabs_n(G, k) \leq |k|^{bn}
   \end{equation*}
   for every finite field $k$ and every positive integer $n$.
\end{lem}
\begin{proof}
 Assume first, that $G$ has UBERG and $R_n(G,\FF_p) \leq p^{en}$ for all primes $p$ and every $n$.
 Let $k$ be a finite field of characteristic $p$.
 Every absolutely irreducible representation of $G$ which is defined over $k$ has a Galois orbit of length at most $[k:\FF_p]$.
 Using Lemma \ref{lem:bijectionGaloisOrbits} we conclude that
 \begin{equation*}
    \rabs_n(G,k) \leq [k : \FF_p] R_{n[k:\FF_p]}(G,\FF_p) \leq [k : \FF_p] p^{en[k:\FF_p]} \leq |k|^{1+en}.
 \end{equation*}
 
 Conversely, assume that $\rabs_n(G, k) \leq |k|^{bn}$ for every finite field $k$.
 By Lemma \ref{lem:bijectionGaloisOrbits} we obtain the inequality
 \begin{equation*}
    r_n(G, \FF_p) \leq \sum_{d | n} \rabs_{n/d}(G, \FF_{p^d})  \leq \sum_{d | n} p^{bn} \leq p^{(1+b)n}.
 \end{equation*}
\end{proof}

\begin{proof}[Proof of Theorem \ref{thm:directProducts}]
Assume $G_1$ and $G_2$ have UBERG.
By Lemma \ref{lem:absolutely} there are constants $b_1, b_2 > 0$ such that
$\rabs_n(G_i,k) \leq |k|^{b_i n}$ for every finite field $k$.

Let $k$ be any finite field.
For absolutely irreducible representations $\pi_i \in \Irrabs_n(G_i,k)$, the outer tensor product representation
$\pi_1 \otimes_k \pi_2$ is absolutely irreducible (see \cite[Thm.\ 2.3]{Fein1967}).
Moreover, every absolutely irreducible representation of $G_1\times G_2$ over $k$ is of this form
(see \cite[Thm.\ 2.7]{Fein1967}).
It follows that
\begin{equation*}
  \rabs_n(G_1\times G_2,k) = \sum_{n_1 n_2 = n} \rabs_{n_1}(G_1,k) \cdot \rabs_{n_2}(G_2,k) \leq n |k|^{(b_1 + b_2)n}
\end{equation*}
and we conclude that $G_1 \times G_2$ has UBERG.
\end{proof}

\subsection{A criterion for UBERG}
\begin{defn}
 We say that a finite group $K$ is \emph{involved} in a profinite group $G$, if 
 it is the continuous quotient of some open subgroup of $G$.
\end{defn}

\begin{thm}
 Let $G$ be a finitely generated profinite group which does not involve some finite group, then $G$ has UBERG.
\end{thm}
\begin{proof}
  By assumption $G$ is $d$-generated for some integer $d$.
  Clearly, if $G$ does not involve some finite group $K$, then
  it does neither involve any alternating groups $\Alt(m)$ for large $m$ nor any classical groups of Lie type of large rank.
  It follows that the set of finite quotients of $G$ is contained in the Babai-Cameron-P\'alfy class $\mathcal{G}(c_0)$ (see \cite{BabaiCameronPalfy1982}) for some large constant $c_0$.
  Corollary 3.3 in \cite{BabaiCameronPalfy1982} implies that there is a constant $c > 0$ with the following property:
  If $\rho\colon G \to \GL_n(\FF_p)$ is an irreducible representation, then $|\rho(G)| \leq p^{cn}$. 
  
  As a next step, we need to know that the number of irreducible subgroups of $\GL_n(\FF_p)$ which are isomorphic to finite quotients of $G$ is small.
  By \cite[Prop.\ 6.1]{JaikinPyber2011} there
  is a constant $c' > 0$ such that the number of conjugacy classes of $d$-generated irreducible subgroups of $\GL_n(\FF_p)$ is at most $p^{c'n}$ for every prime~$p$.
  In total we obtain $r_n(G,\FF_p) \leq p^{c'n}p^{cnd} = p^{(c'+cd)n}$.
\end{proof}

\begin{rmk}
 It is possible to replace \cite[Prop.\ 6.1]{JaikinPyber2011} in the proof by a different argument to obtain a suitable constant $c'$ as in the proof above.
 Indeed, since $G$ is in the class $\mathcal{G}(c_0)$, it follows from a theorem of Borovik, Pyber and Shalev \cite[Theorem~1.3]{BorovikPyberShalev1996} that there is a suitable  constant $c' > 0$.
\end{rmk}

\begin{cor}
 Every finitely generated pro-solvable group has UBERG.
\end{cor}

This results uses deep ingredients. However, 
in order to see that every finitely generated pro-$p$ group has UBERG it suffices to use an elementary argument.
\begin{prop}
 Let $p$ be a prime number. Every finitely generated pro-$p$ group has UBERG. 
\end{prop}
\begin{proof}
 Let $G$ be a $d$-generated pro-$p$ group.
 First observe that $G$ has exactly one irreducible representation over $\FF_p$, namely the trivial representation.
 
 Let $q$ be a prime different from $p$.
 The image of any homomorphism from $G$ into the finite group $\GL_n(\FF_q)$ is a $p$-group and hence is contained in some Sylow-$p$ subgroup.
 Fix some Sylow-$p$ subgroup $S$ of $\GL_n(\FF_q)$.
 We can bound the number of $n$-dimensional irreducible representations of $G$ from above by the number of conjugacy classes of homomorphisms into $\GL_n(\FF_q)$.
 And since all Sylow-$p$ subgroups are conjugate, we obtain the upper bound
 \begin{equation*}
   r_n(G,\FF_q) \leq \left| S \right|^{d}.
 \end{equation*}
 In fact, it follows from the argument in \cite{Weir1955} (with a small extension to the prime $p=2$) that the order of the Sylow-$p$ subgroup is bounded by $|S| \leq p^n q^{pn}$. 
 We conclude that $G$ has UBERG. 
\end{proof}

\subsection{Non-UBERG groups}

\subsubsection{Large groups}

Free groups admit many homomorphisms into any finite group, so the following proposition is not surprising.

\begin{prop}
 Non-abelian free profinite groups do not have UBERG.
\end{prop}
A profinite group $G$ is called \emph{large} if some open subgroup projects onto a non-abelian free profinite group.
\begin{cor}
 Large profinite groups do not have UBERG.
\end{cor}

It is known that non-abelian free groups have super-exponential subgroup growth (see Chapter 2 in \cite{SubgroupGrowth}).
Henceforth the proposition follows from Proposition \ref{prop:UBERGimpliesESG}.
However, the result follows also immediately from the next elementary lemma.

\begin{lem}\label{lem:freeNotPFR}
 Let $F$ be a free group on $d\geq 2 $ generators. For every prime $p$ the number of irreducible $\FF_p$-representations of $F$ satisfies
 \begin{equation*}
      r_n(F,\FF_p) \geq c_p^d p^{n^2(d-1)}.
 \end{equation*}
 where $c_p = (1 - \frac{1}{p} - \frac{1}{p^2}) \geq \frac{1}{4}$.
\end{lem}
\begin{proof}
We will proceed by counting certain homomorphisms of $F$ into $\GL_n(\FF_p)$.
We will use that for any finite group $S$, the number of homomorphisms from $F$ into $S$ is $|S|^d$.

Let $n = n_1 + n_2$ be a partition of $n$ and let $P(n_1,n_2,\FF_p)$ denote the associated maximal parabolic subgroup of $\GL_n(\FF_p)$;
that is 
\begin{equation*}
   P(n_1,n_2,\FF_p) = \begin{pmatrix}
                                \GL_{n_1}(\FF_p) & * \\
                                 0 & \GL_{n_2}(\FF_p) \\
                             \end{pmatrix} .
\end{equation*}
Let $\rho\colon F \to \GL_n(\FF_p)$ be a homomorphism and assume that $\rho$ defines a reducible representation of $F$.
In this case there are $n_1, n_2 \geq 1$ such that $\rho$ is conjugate to a homomorphism into $P(n_1,n_2,\FF_p)$.
Conversely, every homomorphism of $F$ into some maximal parabolic subgroup defines a reducible representation. Observe further, that every such homomorphism is
conjugate to at most $\frac{|\GL_n(\FF_p)|}{|P(n_1,n_2,\FF_p)|}$ homomorphisms whose image is not contained in $P(n_1,n_2,\FF_p)$.
Since every homomorphism from $F$ to $\GL_n(\FF_p)$ which defines an irreducible representation is conjugate to at most $|\GL_n(\FF_p)|$ others, we conclude
\begin{equation}\label{eq:freeirrednumber}
   r_n(F,\FF_p) \geq |\GL_n(\FF_p)|^{d-1} - \sum_{k = 1}^{n-1} |P(k,n-k,\FF_p)|^{d-1}.
\end{equation}
To finish the proof we will use the following simple estimates. One can check (using only Boole's inequality and the geometric series) that
\begin{equation*} 
   |\GL_n(\FF_p)| = p^{n^2} \prod_{i=1}^n (1-\frac{1}{p^i}) \geq  p^{n^2} c_p.
\end{equation*}
Further, we observe that
\begin{align*}
   |P(k,n-k,\FF_p)| &= p^{n^2 - k(n-k)} \prod_{i=1}^k(1-\frac{1}{p^i}) \prod_{i=1}^{n-k}(1-\frac{1}{p^i}) \\
    &\leq p^{n^2 - k(n-k)} (1-\frac{1}{p})^2 \leq p^{n^2 - k(n-k)} c_p.
\end{align*}
We use these bounds in equation \eqref{eq:freeirrednumber} to obtain
\begin{align*}
   r_n(F,\FF_p) &\geq c_p^{d-1} p^{n^2(d-1)} \left( 1- \sum_{k=1}^{n-1} p^{-k(n-k)(d-1)} \right)\\
   &\geq c_p^{d-1} p^{n^2(d-1)} \left( 1 - (n-1)p^{-(n-1)(d-1)} \right)\\
   &\geq c_p^{d-1} p^{n^2(d-1)} \left( 1 - p^{-(n-1)(d-2) - 1} \right)
   \geq c_p^d p^{n^2(d-1)}   
\end{align*}
where we use $n-1 \leq p^{n-2}$ going from the second to the third line.
\end{proof}

\section{Weakly positively finitely related groups}\label{sec:weaklyPFR}

In this section we briefly discuss a mysterious property of groups, namely \emph{weakly PFR}.
The following observation is a consequence of Theorem~\ref{thm:PFRisUBERG}.

\begin{prop}\label{prop:d+1PFR}
 Let $G$ be a finitely presented profinite group with $d = d(G)$ generators.
 If the $d+1$-generator presentation
 \begin{equation*}
  1 \longrightarrow R \longrightarrow F_{d+1} \longrightarrow G \longrightarrow 1
 \end{equation*}
 is PFR, then $G$ is PFR.
\end{prop}
\begin{proof}
  For simplicity we write $F = F_{d+1}$.
  Let $p$ be a prime and let $(\rho,V_\rho)$ be an irreducible representation of $G$ over $\FF_p$ of dimension $k$.
  The semidirect product $E_V = V_\rho \rtimes G$ is $(d+1)$-generated (see Lemma \ref{lem:minimalExtensionsBasic}) 
  and yields a minimal extension of $G$. Every such split extension can be obtained from some maximal $F$-stable
  subgroup $M_\rho$ of $R$. In particular, $m_{p^k}^{F}(R) \geq r_k(G,\FF_p)$.
  Proposition \ref{prop:PFRiffPMNSG} implies that $G$ has UBERG and thus Theorem \ref{thm:PFRisUBERG} implies the proposition.
\end{proof}

We conclude that PFR, which by definition is a property of all presentations of a group, can be read off from a single presentation.
However, the presentation is \emph{not} the one with the minimal number of generators but the one with one extra generator. This naturally leads to the following definition.

\begin{defn}
A finitely generated profinite group $G$ with $d = d(G)$ is called \emph{weakly PFR} if the minimal presentation
\begin{equation*}
  1 \longrightarrow R \longrightarrow F_{d} \longrightarrow G \longrightarrow 1
 \end{equation*}
 is PFR. 
\end{defn}

Clearly, a PFR group is weakly PFR. The free profinite group $F_d$ on $d\geq 2$ generators is clearly weakly PFR.
However, by Lemma \ref{lem:freeNotPFR} the group $F_d$ is not PFR.
Weakly PFR is a very fragile property. The next proposition shows that the group $C_m \times F_2$ is not weakly PFR. Hence, this property it is not preserved by direct products and is not a property of the commensurability class. The reason might be that the minimal number of generators is a very intricate invariant.
\begin{prop}\label{prop:weaklyPFRandDirectProducts}
Let $G,H$ be non-trivial finitely generated profinite groups. Assume that $G$ does not have UBERG and $d(G\times H) = d(G) + d(H)$. Then $G \times H$ is not weakly PFR.
\end{prop}
The following lemma will be used to prove the proposition.
\begin{lem}\label{lem:directProductsGenerators}
Let $G$ and $H$ be non-trivial finitely generated profinite groups and let $(\rho, V_\rho)$ be an irreducible non-trivial representation of $G$ over some finite field $\FF_p$.
Then $d( (V_\rho \rtimes G) \times H) \leq d(G) + d(H)$.
\end{lem}

\begin{rmk}
 Let $G, H, \rho$ and $V_\rho$ be as in the statement of the lemma. We can define a representation of $G\times H$ on $V_\rho$ by letting $H$ act trivially. With this action $V_\rho \rtimes (G\times H)$ is isomorphic to $(V_\rho \rtimes G) \times H$. Hence there is no confusion in writing $V_\rho \rtimes G\times H$.
\end{rmk}
\begin{proof}[Proof of Lemma \ref{lem:directProductsGenerators}]
 Let $n = d(G)$ and $m = d(H)$. By assumption $n,m \geq 1$.
 Let $g_1,\dots,g_n$ and $h_1,\dots,h_m$ be minimal generating sets of $G$ and $H$ respectively.
 We choose some non-zero vector $v \in V_\rho$. We verify that the elements $(0,g_1,1), \dots, (0,g_n,1), (v,1,h_1), (0,1,h_2), \dots, (0,1,h_m)$ generate
 the group $V_\rho \rtimes G \times H$.
 Clearly, the group generated by these elements projects onto $G\times H$. It suffices to check that it contains all elements from $V_\rho$.
 The representation $(\rho, V_\rho)$ of $G$ is non-trivial, hence some element (say $g_1$) acts non-trivially on $v$.
 The commutator $[(0,g_1,1), (v,1,h_1)] = (\rho(g_1)v - v, 1, 1)$ is a non-trivial element in $V_\rho$. Since $V_\rho$ is irreducible, the vector $\rho(g_1)v - v$  generates $V_\rho$ under the action of $G$.
\end{proof}

\begin{proof}[Proof of Proposition \ref{prop:weaklyPFRandDirectProducts}]
Let $1 \to R \to F \stackrel{\pi}{\longrightarrow} G \times H \to 1$ be the minimal presentation of $G\times H$ where $F$ is the free group on $d(G)+d(H)$ elements.
Let $p$ be a prime and let $(\rho, V_\rho)$ be a non-trivial irreducible representation of $G$ over $\FF_p$.
By Lemma \ref{lem:directProductsGenerators} the group $V_\rho \rtimes G \times H$ is $d(G) + d(H)$ generated. Hence, by Gasch\"utz' Lemma,
there is a surjective homomorphism $f \colon F \to V_\rho \rtimes G \times H$ such that the following diagram commutes.
\begin{equation*}
   \begin{CD} 
       R @>>> F @>{\pi}>> G\times H \\
       @V{f_{|R}}VV @V{f}VV @| \\
       V_\rho @>>> V_\rho \rtimes G \times H @>>> G\times H
   \end{CD}
\end{equation*}
The kernel of $f$ is a maximal $F$-stable subgroup of $R$.
We conclude (for $k>1$) that $m^F_{p^k}(R) \geq r_k(G, \FF_p)$. Hence, if $G$ does not have UBERG, then $R$ does not have polynomial maximal $F$-stable subgroup growth and
the presentation is not PFR by Proposition \ref{prop:PFRiffPMNSG}.
\end{proof}

The finitely generated free profinite groups are the only examples of weakly PFR (but not PFR) groups we are aware of. Therefore we conclude this section with the following question.
\begin{question}
 Is there a profinite group which is weakly PFR, but is neither PFR nor free?
\end{question}